\documentclass[11pt]{article}

\usepackage{amsmath, amssymb}
\usepackage{tikz}

\usepackage{authblk}
\usepackage{geometry}
\usepackage{hyperref}
\usepackage{graphicx}
\usepackage{epstopdf}
\usepackage{mathrsfs}
\usepackage{amsfonts}
\usepackage{amscd,amsthm,bm,psfrag}
\usepackage[numbers,sort&compress]{natbib}

\makeatletter
\renewcommand{\@seccntformat}[1]{{\csname the#1\endcsname}{\normalsize.}\hspace{.5em}}

\makeatother

\numberwithin{equation}{section}

\def \[{\begin{equation*}}
\def \]{\end{equation*}}

\newtheorem{thm}{Theorem}[section]

\newtheorem{lem}[thm]{Lemma}

\newtheorem{prop}[thm]{Proposition}

\newtheorem{con}[thm]{Conjecture}

\newtheorem{cla}{Claim}

\newtheorem*{thm*}{Theorem}
\newtheorem*{prop*}{Proposition}

\newcommand\cC{{\mathcal C}}

\newcommand\cK{{\mathcal K}}

\newcommand\cN{{\mathcal N}}

\newcommand\ex{\ensuremath{\mathrm{ex}}}

\def\final{0}  
\def\iflong{\iffalse}
\ifnum\final=0  
\newcommand{\jnote}[1]{{\color{blue}[{\tiny \textbf{Junpeng:} \bf #1}]\marginpar{\color{blue}*}}}
\newcommand{\ynote}[1]{{\color{purple}[{\tiny \textbf{Yuhang:} \bf #1}]\marginpar{\color{purple}*}}}
\else 
\newcommand{\jnote}[1]{}
\newcommand{\ynote}[1]{}
\fi

\begin{document}

\title{Counting large cliques in graphs with a forbidden tree
}

\author[a,b,*]{Junpeng Zhou\,$^{\rm a,b}$, \ Xiying Yuan\,}

\affil[a]{\small Department of Mathematics, Shanghai University, Shanghai 200444, PR China}
\affil[b]{\small Newtouch Center for Mathematics of Shanghai University, Shanghai 200444, PR China}

\date{}

\maketitle

\footnotetext{*\textit{Corresponding author}.}
\footnotetext{Email addresses: \texttt{junpengzhou@shu.edu.cn} (J. Zhou), \texttt{xiyingyuan@shu.edu.cn} (X. Yuan).}

\begin{abstract}
Given graphs $H$ and $F$, the generalized Tur\'{a}n number $\ex(n,H,F)$ is the maximum number of copies of $H$ in an $n$-vertex $F$-free graph. Alon and Shikhelman (J. Combin. Theory Ser. B, 2016) initiated the systematic study of generalized Tur\'{a}n problems. Let $T$ be a tree on $k$ vertices, and write $n=a(k-1)+b$, where $0\leq b<k-1$. Recently, Gerbner and Palmer (Electron. J. Combin., 2026) proposed the following conjecture: for every $r\geq3$, the graph $aK_{k-1}\cup K_b$ maximizes the number of copies of $K_r$ among all $n$-vertex $T$-free graphs. In this paper, we verify their conjecture when $r=k-2$ or $r=k-3\geq5$. More precisely, we show that $\ex(n,K_r,T)=a\binom{k-1}{r}+\binom{b}{r}$ and characterize all extremal graphs.
\end{abstract}


{\noindent{\bf AMS (2020) subject classifications:} 05C35, 05C05}

\section{\normalsize Introduction}
Throughout this paper, all graphs are finite, simple and undirected. Given graphs $H$ and $F$, a graph $G$ is called \textit{$F$-free} if it does not contain a copy of $F$ as a subgraph. Let $\cN(H,G)$ denote the number of copies of $H$ in $G$. Let
$$
 \ex(n,H,F)=\max\bigl\{\cN(H,G): |V(G)|=n \text{ and } G \text{ is } F\text{-free}\bigr\}.
$$
Alon and Shikhelman \cite{AlSh} initiated the systematic study of the function $\ex(n,H,F)$, which is often called the \textit{generalized Tur\'{a}n problem}. When $H=K_2$, the generalized Tur\'{a}n number reduces to the classical Tur\'{a}n number, that is, $\ex(n,K_2,F)=\ex(n,F)$.
For a recent survey on generalized Tur\'{a}n problems, one can refer to the work of Gerbner and Palmer \cite{GePa}.

In this paper, we consider the case where $H$ is a clique and $F$ is a tree. Let $T$ be a tree on $k$ vertices. The Erd\H{o}s--S\'{o}s conjecture \cite{Er} states that $ \ex(n,T)\leq (k-2)n/2$. When $k-1$ divides $n$, the disjoint union of copies of $K_{k-1}$ shows that the above bound is sharp. Gerbner, Methuku and Palmer \cite{GeMePa} proved that if the Erd\H{o}s--S\'{o}s conjecture holds for $T$ and its subtrees, then
$$
 \ex(n,K_r,T)\leq \frac{n}{k-1}\binom{k-1}{r}
$$
for $3\leq r\leq k-3$. This bound is also sharp when $k-1$ divides $n$. Letzter \cite{Le} determined the order of magnitude of $\ex(n,H,T)$ for arbitrary fixed graphs $H$ and trees $T$.

Several exact results are known for special trees.
When $T$ is a star, the result follows from the Gan--Loh--Sudakov conjecture~\cite{GaLoSu}, proved by Chase~\cite{Ch} and later given a unified proof by Chao and Dong~\cite{ChDo}. When $T$ is a path, Luo~\cite{Lu} established the corresponding upper bound. Later, Chakraborti and Chen~\cite{ChCh} determined the exact value and characterized all extremal graphs. Furthermore, Gerbner~\cite{GeDS} obtained exact results for double stars.

In \cite{GePa}, Gerbner and Palmer observed that for $n=a(k-1)+b$ with $0\leq b<k-1$, the natural analogue of the extremal construction above is $aK_{k-1}\cup K_b$. For $3\leq r\leq k-1$, each copy of $K_{k-1}$ contains $\binom{k-1}{r}$ copies of $K_r$. This suggests that the generalized Tur\'an problem may be easier for $r>2$ than the classical case $r=2$. Motivated by this observation, they proposed the following conjecture.

\begin{con}[Gerbner and Palmer~\cite{GePa}]\label{con:GP}
Let $r\geq3$ and $T$ be a tree on $k$ vertices. For $n=a(k-1)+b$ with $0\leq b<k-1$,
$$
 \ex(n,K_r,T)=\cN\bigl(K_r,aK_{k-1}\cup K_b\bigr)
 =a\binom{k-1}{r}+\binom{b}{r}.
$$
\end{con}

For $r\geq k$, we have $\ex(n,K_r,T)=0$, since $T\subseteq K_r$. The case $r=k-1$ is also immediate. Indeed, every copy of $K_{k-1}$ in a $T$-free graph must be a connected component. Otherwise, an edge joining the clique to a vertex outside it would yield a copy of $T$. Now we consider the next two cases, namely $r=k-2$ and $r=k-3$. In both cases, we verify the conjecture and characterize all extremal graphs.

\begin{thm}\label{thm:main}
Let $r\geq3$ and $T$ be a tree on $k$ vertices. For $n=a(k-1)+b$ with $0\leq b<k-1$, we have
\begin{equation*}
 \ex(n,K_r,T)=a\binom{k-1}{r}+\binom{b}{r}
\end{equation*}
when $r=k-2\geq3$ or $r=k-3\geq5$. Moreover, the extremal graphs are characterized as follows.
\begin{itemize}
 \item If $b\geq r$, then $aK_{k-1}\cup K_b$ is the unique extremal graph up to isomorphism.
 \item If $b<r$, then the extremal graphs are precisely the graphs of the form $aK_{k-1}\cup H$, where $H$ is an arbitrary graph on $b$ vertices.
\end{itemize}
\end{thm}

The main idea is to consider the $r$-uniform hypergraph $\cK_r(G)$ whose hyperedges are the vertex sets of copies of $K_r$ in $G$. We show that every connected component of $\cK_r(G)$ has at most $k-1$ vertices. Thus, the desired bound follows from Lemma~\ref{lem:packing}.

The rest of this paper is organized as follows. In Section \ref{sec:prelim}, we establish several preliminary lemmas. The proof of Theorem~\ref{thm:main} will be presented in Section \ref{sec:proof}. Finally, we conclude with several remarks in Section \ref{sec:concluding}.

\section{\normalsize Preliminary lemmas}\label{sec:prelim}

For a graph $G$ and a vertex $x\in V(G)$, let $N_G(x)$ denote the \textit{neighborhood} of $x$ in $G$. For a vertex set $U\subseteq V(G)$, let $G[U]$ be the subgraph of $G$ induced by $U$ and $G-U:=G[V(G)\setminus U]$. For $x\in V(G)\setminus U$, set $N_U(x)=N_G(x)\cap U$.

First, we prove the following inequality.

\begin{lem}\label{lem:packing}
Let $q>r\geq2$ and $v_1,\ldots,v_m$ be nonnegative integers with $v_i\leq q$. Suppose that $\sum_{i=1}^m v_i\leq aq+b$, where $0\leq b<q$. Then
\begin{equation*}
 \sum_{i=1}^m\binom{v_i}{r}\leq a\binom{q}{r}+\binom{b}{r}.
\end{equation*}
If the equality holds, then exactly $a$ of the $v_i$ are equal to $q$. After these terms are removed, the following statements hold.
\begin{itemize}
 \item If $b\geq r$, then exactly one remaining term is positive, and it is equal to $b$.
 \item If $b<r$, then every remaining term is smaller than $r$.
\end{itemize}
\end{lem}
\begin{proof}[\bf Proof]
Let $d=aq+b-\sum_{i=1}^m v_i$. We add $d$ parts of size one. Since $r\geq2$, these additional parts do not change the sum of the binomial coefficients. Thus, we may assume that the sum of all parts is $aq+b$.

Suppose that $1\leq x\leq y<q$ are two parts. We replace
them by $x-1$ and $y+1$, respectively. Since
$$
 \binom{x-1}{r}+\binom{y+1}{r}-\binom{x}{r}-\binom{y}{r}=
 \binom{y}{r-1}-\binom{x-1}{r-1}\geq0,
$$
this operation does not decrease the sum. Repeating this operation, we obtain $a$ parts of size $q$, one part of size $b$ when $b>0$, and all remaining parts of size zero. This proves the upper bound.

Now suppose that the equality holds. If fewer than $a$ of the original parts have size $q$, then a new part of size $q$ must be obtained during the above process. At this step, two parts of sizes $x$ and $q-1$ are replaced by parts of sizes $x-1$ and $q$, respectively. Since $x-1\leq q-2$ and $q>r$, the increase is $\binom{q-1}{r-1}-\binom{x-1}{r-1}>0$, which contradicts the equality assumption. Thus, at least $a$ of the original parts have size $q$. Since their total size is at most $aq+b<(a+1)q$, there cannot be more than $a$ such parts. Hence, exactly $a$ of the original parts have size $q$.

After these parts are removed, the remaining original parts have total size at most $b$. Suppose that $b\geq r$. If the remaining original parts do not consist of a single part of size $b$, then, after the additional parts of size one are included, there are at least two positive parts with total size $b$. In the last step that produces a part of size $b$, the parts of sizes $1$ and $b-1$ are replaced by parts of sizes $0$ and $b$. The increase is $\binom{b}{r}-\binom{b-1}{r}=\binom{b-1}{r-1}>0$, which contradicts the equality assumption. Therefore, exactly one remaining original part is positive, and it is equal to $b$.

If $b<r$, then the remaining original parts have total size at most $b<r$. Thus, every remaining part has size less than $r$. This completes the proof.
\end{proof}

We also need the following lemma on trees. A vertex $v$ of a tree $T$ is called a \textit{centroid} of $T$ if every component of $T-\{v\}$ has at most $\lfloor |V(T)|/2\rfloor$ vertices. Jordan \cite{Jor} proved that a centroid always exists.

\begin{lem}\label{lem:separation}
Let $T$ be a tree on $k$ vertices.
\begin{itemize}
 \item[\textbf{(i)}] If $k\geq5$, there is a vertex $u\in V(T)$ such that the components of $T-\{u\}$ can be divided into two families, each containing at most $k-3$ vertices in total.
 \item[\textbf{(ii)}] If $k\geq8$, there is a vertex $u\in V(T)$ such that the components of $T-\{u\}$ can be divided into two families, each containing at most $k-4$ vertices in total.
\end{itemize}
\end{lem}
\begin{proof}[\bf Proof]
Choose a centroid $u$ of $T$. Then every component of $T-\{u\}$ has order at most $\lfloor k/2\rfloor$.

If $k\geq 5$, then every component of $T-\{u\}$ has order at most $k-3$. If some component has at least two vertices, then let this component form the first family and let all remaining components form the second family. Otherwise, every component has order one, and we take any two components as the first family and all remaining components as the second family. In both cases, the first family contains between $2$ and $k-3$ vertices. Since the components of $T-\{u\}$ have total order $k-1$, the second family also contains between $2$ and $k-3$ vertices.

If $k\geq8$, then every component of $T-\{u\}$ has order at most $k-4$. If some component has at least three vertices, then let this component form the first family and let all remaining components form the second family. Otherwise, every component has order one or two. Since the components have total order $k-1\geq7$, we may choose a collection of components with total order three or four to form the first family and let all remaining components form the second family.
In both cases, the first family contains between $3$ and $k-4$ vertices. Since the components of $T-\{u\}$ have total order $k-1$, the same holds for the second family. This completes the proof.
\end{proof}

Next, we prove three tree embedding lemmas. We first consider two intersecting cliques.

\begin{lem}\label{lem:two-cliques}
Let $T$ be a tree on $k$ vertices. Let $V_1$ and $V_2$ be the vertex sets of two copies of $K_r$ in a graph $G$. Suppose that $V_1\cap V_2\neq\emptyset$, $|V_1\cup V_2|\geq k$, and $r=k-2\geq3$ or $r=k-3\geq5$. Then $G[V_1\cup V_2]$ contains a copy of $T$.
\end{lem}
\begin{proof}[\bf Proof]
By Lemma~\ref{lem:separation}, we may choose a vertex $u\in V(T)$ and divide the components of $T-\{u\}$ into two families. Let $x\in V_1\cap V_2$ and identify $u$ with $x$. If $r=k-2\geq3$, then each family contains at most $k-3=r-1$ vertices. If $r=k-3\geq5$, then each family contains at most $k-4=r-1$ vertices.

Let $s_i$ denote the number of vertices contained in the $i$th family. Then $s_i\leq r-1$ and $s_1+s_2=k-1$. We first identify the vertices in the first family with distinct vertices of $V_1\setminus\{x\}$, choosing vertices from $V_1\setminus V_2$ whenever possible. If $s_1\leq |V_1\setminus V_2|$, then no vertex of $V_2\setminus\{x\}$ is used, and at least $r-1\geq s_2$ vertices remain available. If $s_1>|V_1\setminus V_2|$, then exactly $s_1-|V_1\setminus V_2|$ vertices of $V_2\setminus\{x\}$ are used. Hence, the number of unused vertices in $V_2\setminus\{x\}$ is
$$
 r-1-(s_1-|V_1\setminus V_2|)=|(V_1\cup V_2)\setminus\{x\}|-s_1 \geq k-1-s_1=s_2.
$$
Therefore, we may identify the vertices in the second family with distinct unused vertices of $V_2\setminus\{x\}$. Since $G[V_1]\cong G[V_2]\cong K_r$, the above identification gives a copy of $T$ in $G[V_1\cup V_2]$. This completes the proof.
\end{proof}

Now we consider a clique together with two or three additional vertices.

\begin{lem}\label{lem:two-outside}
Let $r\geq3$ and $V_0$ be the vertex set of a copy of $K_r$ in a graph $G$. Let $x_1,x_2\in V(G)\setminus V_0$ be distinct. If $|N_{V_0}(x_i)|\geq r-1$ for $i\in\{1,2\}$, then $G[V_0\cup\{x_1,x_2\}]$ contains every tree on $r+2$ vertices.
\end{lem}
\begin{proof}[\bf Proof]
Let $T$ be a tree on $r+2$ vertices. We choose two leaves $z_1,z_2$ of $T$ and identify $z_1,z_2$ with $x_1,x_2$, respectively. Note that
$$
 \bigl|N_{V_0}(x_1)\cap N_{V_0}(x_2)\bigr| \geq |N_{V_0}(x_1)|+|N_{V_0}(x_2)|-|V_0| \geq r-2\geq1.
$$
If the two leaves have the same neighbor, then we may identify their common neighbor with a vertex of $N_{V_0}(x_1)\cap N_{V_0}(x_2)$.
If the two leaves have distinct neighbors, then we choose distinct vertices $y_i\in N_{V_0}(x_i)$ for $i\in\{1,2\}$ and identify the neighbor of $z_i$ with $y_i$. Such a choice exists since $|N_{V_0}(x_i)|\geq r-1\geq2$ for $i\in\{1,2\}$.

In either case, we greedily identify the remaining vertices of $T$ with distinct unused vertices of $V_0$. Since $G[V_0]\cong K_r$, the above identification gives a copy of $T$ in $G[V_0\cup\{x_1,x_2\}]$. This completes the proof.
\end{proof}

\begin{lem}\label{lem:three-outside}
Let $r\geq5$ and $V_0$ be the vertex set of a copy of $K_r$ in a graph $G$. Let $x_1,x_2,x_3\in V(G)\setminus V_0$ be distinct. Suppose that $|N_{V_0}(x_i)|\geq r-2$ for every $i\in[3]$ and $\bigl|N_{V_0}(x_1)\cap N_{V_0}(x_2)\cap N_{V_0}(x_3)\bigr|\geq2$. Then $G[{V_0}\cup\{x_1,x_2,x_3\}]$ contains every tree on $r+3$ vertices.
\end{lem}
\begin{proof}[\bf Proof]
Let $T$ be a tree on $r+3$ vertices. First, suppose that $T=P_{r+3}$. We choose distinct vertices $c_1,c_2\in N_{V_0}(x_1)\cap N_{V_0}(x_2)\cap N_{V_0}(x_3)$. Since $|N_{V_0}(x_3)|\geq r-2\geq3$, we may choose a vertex $w\in N_{V_0}(x_3)\setminus\{c_1,c_2\}$. Then $x_1,c_1,x_2,c_2,x_3,w$ followed by the vertices of $V_0\setminus\{c_1,c_2,w\}$ in any order, forms a Hamiltonian path $P_{r+3}$ in $G[V_0\cup\{x_1,x_2,x_3\}]$.

Now suppose that $T$ is not a path. Then $T$ has at least three leaves. We choose three leaves $z_1,z_2,z_3$ and identify them with $x_1,x_2,x_3$, respectively. If the three leaves have a common neighbor, then we identify this vertex with a vertex of $N_{V_0}(x_1)\cap N_{V_0}(x_2)\cap N_{V_0}(x_3)$.
If the three leaves have exactly two distinct neighbors, then we may identify these two neighbors with two distinct vertices of $N_{V_0}(x_1)\cap N_{V_0}(x_2)\cap N_{V_0}(x_3)$.
If the three leaves have three distinct neighbors, then we greedily choose distinct vertices $y_i\in N_{V_0}(x_i)$ for $i\in[3]$ and identify the neighbor of $z_i$ with $y_i$. Such a choice exists since $|N_{V_0}(x_i)|\geq r-2\geq3$ for every $i\in[3]$.

In either case, we greedily identify the remaining vertices of $T$ with distinct unused vertices of $V_0$. Since $G[V_0]\cong K_r$, the above identification gives a copy of $T$ in $G[V_0\cup\{x_1,x_2,x_3\}]$. This completes the proof.
\end{proof}

\section{\normalsize Proof of Theorem \ref{thm:main}}\label{sec:proof}
Now we prove Theorem \ref{thm:main}. For the lower bound, the graph $aK_{k-1}\cup K_b$ is $T$-free and contains exactly $a\binom{k-1}{r}+\binom{b}{r}$ copies of $K_r$.

We now consider the upper bound. Let $G$ be an $n$-vertex $T$-free graph. Denote by $\cK_r(G)$ the family of vertex sets of copies of $K_r$ in $G$. Define an auxiliary graph $\Gamma(G)$ with vertex set $\cK_r(G)$, where two members are adjacent if they intersect. We call the connected components of $\Gamma(G)$ the \textit{$K_r$-components} of $G$. For a $K_r$-component $\cC$, let $U(\cC)=\bigcup_{A\in V(\cC)}A$.
Clearly, the sets $U(\cC)$ corresponding to distinct $K_r$-components are vertex-disjoint.

\begin{cla}\label{claim1}
    $|U(\cC)|\leq k-1$ for every $K_r$-component $\cC$ of $G$.
\end{cla}
\begin{proof}[\bf Proof of Claim.]
Suppose to the contrary that $|U(\cC)|\geq k$ for some $K_r$-component $\cC$. We distinguish the following two cases.

\medskip
\noindent {\bf Case 1.} $r=k-2$ and $k\geq5$.
\medskip

Since $G$ is $T$-free, Lemma~\ref{lem:two-cliques} implies that if two distinct members $A,B\in\cK_r(G)$ intersect, then $|A\cup B|\leq k-1=r+1$, and thus $|A\cup B|=r+1$. Therefore,
\begin{equation}\label{eq:int-one}
    |A\cap B|=|A|+|B|-|A\cup B|=r-1.
\end{equation}

We claim that any two distinct members in the same $K_r$-component intersect in $r-1$ vertices. Indeed, let $A,B,C\in\cK_r(G)$ be three distinct members such that $A\cap B\neq\emptyset$ and $B\cap C\neq\emptyset$. By \eqref{eq:int-one}, $|A\cap B|=|B\cap C|=r-1$. Hence, $|A\cap C|\geq r-2\geq1$. Then $|A\cap C|=r-1$ by \eqref{eq:int-one}.
Now let $A_0,A_1,\ldots,A_\ell$ be a path in $\Gamma(G)$. Repeatedly applying the above argument shows that $|A_0\cap A_i|=r-1$ for every $i\in[\ell]$, and we are done.

Fix $A\in V(\cC)$. Since $|A|=r=k-2$, there are two distinct vertices $x_1,x_2\in U(\cC)\setminus A$. For each $i\in[2]$, choose $B_i\in V(\cC)\setminus\{A\}$ containing $x_i$. By the above claim, $|A\cap B_i|=r-1$. Hence, $B_i=\{x_i\}\cup(A\cap B_i)$ and $|N_A(x_i)|\geq r-1$. By Lemma~\ref{lem:two-outside}, we may find a copy of $T$ in $G$, which is a contradiction. Thus, $|U(\cC)|\leq k-1$ for every $K_r$-component $\cC$ of $G$.

\medskip
\noindent {\bf Case 2.} $r=k-3$ and $k\geq8$.
\medskip

Since $G$ is $T$-free, Lemma~\ref{lem:two-cliques} implies that if two distinct members $A,B\in\cK_r(G)$ intersect, then $|A\cup B|\leq k-1=r+2$. Therefore,
\begin{equation}\label{eq:int-two}
 |A\cap B|=|A|+|B|-|A\cup B|\geq r-2.
\end{equation}

We claim that any two distinct members in the same $K_r$-component intersect in at least $r-2$ vertices. Indeed, let $A,B,C\in\cK_r(G)$ be three distinct members such that $A\cap B\neq\emptyset$ and $B\cap C\neq\emptyset$. By~\eqref{eq:int-two}, we have $|A\cap B|\geq r-2$ and $|B\cap C|\geq r-2$. Hence, $|A\cap C|\geq r-4\geq1$. Then $|A\cap C|\geq r-2$ by~\eqref{eq:int-two}. Now let $A_0,A_1,\ldots,A_\ell$ be a path in $\Gamma(G)$. Repeatedly applying the above argument shows that $|A_0\cap A_i|\geq r-2$ for every $i\in[\ell]$, and we are done.

Fix $A\in V(\cC)$. For each $B\in V(\cC)\setminus\{A\}$, let $L_B:=B\setminus A$ and $R_B:=A\setminus B$. Clearly, $|L_B|=|R_B|$. By the above claim, $|L_B|=|R_B|\leq2$. Since $|U(\cC)|\geq k=r+3$ and $|A|=r$, the union of all sets $L_B$ contains at least three vertices. This also implies that $|V(\cC)\setminus \{A\}|\geq2$.

Suppose that $|L_B\cup L_C|\geq3$ for some $B,C\in V(\cC)\setminus\{A\}$. Since $|B\cap C|\geq r-2$, we have
$$
 r-|R_B\cup R_C|+|L_B\cap L_C| =|B\cap C|\geq r-2.
$$
As $|L_B|,|L_C|\leq2$ and $|L_B\cup L_C|\geq3$, we have $|L_B\cap L_C|\leq1$. It follows that $|R_B\cup R_C|\leq3$.
Choose distinct vertices $x_1,x_2,x_3\in L_B\cup L_C$. Then each $x_i$ belongs to $L_B=B\setminus A$ or $L_C=C\setminus A$, and hence $|N_A(x_i)|\geq r-2$ for every $i\in[3]$. Moreover, all three vertices are adjacent to every vertex of $A\setminus(R_B\cup R_C)$. Since $|R_B\cup R_C|\leq3$,
$$
 \bigl|N_A(x_1)\cap N_A(x_2)\cap N_A(x_3)\bigr|\geq |A|-|R_B\cup R_C|\geq r-3\geq2.
$$
By Lemma~\ref{lem:three-outside}, we may find a copy of $T$ in $G$, which is a contradiction.

It remains to consider the case where $|L_B\cup L_C|\leq2$ for every $B,C\in V(\cC)\setminus\{A\}$. Since $B\neq A$ and $|A|=|B|$, we have $L_B\neq\emptyset$. We claim that $|L_B|=1$ for every $B\in V(\cC)\setminus\{A\}$. Otherwise, assume that $|L_B|=2$ for some $B\in V(\cC)\setminus\{A\}$. Since the union of all sets $L_{B'}$ contains at least three vertices, there is some $C\in V(\cC)\setminus\{A,B\}$ such that $L_C$ contains a vertex outside $L_B$. Hence, $|L_B\cup L_C|\geq3$, which is a contradiction.
Thus, we may choose $B_1,B_2,B_3\in V(\cC)\setminus\{A\}$ such that $L_{B_i}=\{x_i\}$ for three distinct vertices $x_1,x_2,x_3$. Recall that $|L_B|=|R_B|$ for each $B\in V(\cC)\setminus\{A\}$. Then $|R_{B_i}|=1$. Hence, $B_i=\{x_i\}\cup(A\cap B_i)$ and $|N_A(x_i)|\geq r-1$ for every $i\in[3]$. Moreover,
$$
 \bigl|N_A(x_1)\cap N_A(x_2)\cap N_A(x_3)\bigr|
 \geq \left|A\setminus\bigcup_{i=1}^3R_{B_i}\right|
 \geq r-3\geq2.
$$
By Lemma~\ref{lem:three-outside}, we may again find a copy of $T$ in $G$, which is a contradiction. Thus, $|U(\cC)|\leq k-1$ for every $K_r$-component $\cC$ of $G$.
\end{proof}

Now let $\cC_1,\ldots,\cC_m$ be the $K_r$-components of $G$. Set $|U(\cC_i)|=:n_i$. By Claim~\ref{claim1}, $n_i\leq k-1$ for every $i\in[m]$. Moreover, we have $\sum_{i=1}^m n_i\leq n$. Since $\cC_i$ contains at most $\binom{n_i}{r}$ copies of $K_r$, by Lemma~\ref{lem:packing}, we obtain
\begin{equation}\label{eq:upper}
\cN(K_r,G)=\sum_{i=1}^m|\cC_i|\leq\sum_{i=1}^m\binom{n_i}{r} \leq a\binom{k-1}{r}+\binom{b}{r}.
\end{equation}
This proves the upper bound.

Finally, we characterize the equality case. Suppose that the equality holds in~\eqref{eq:upper}. Then equality holds in both inequalities in~\eqref{eq:upper}. By Lemma~\ref{lem:packing}, there are exactly $a$ $K_r$-components containing $k-1$ vertices of $G$. Moreover, each of these components contains all $\binom{k-1}{r}$ copies of $K_r$ on its $k-1$ vertices. Hence, the $k-1$ vertices contained in each such component induce a copy of $K_{k-1}$.

We claim that every copy of $K_{k-1}$ is a connected component of $G$. Otherwise, assume that $G[Q]\cong K_{k-1}$ for $Q\subseteq V(G)$, and $xy\in E(G)$ for some $x\notin Q$ and $y\in Q$. Choose a leaf $z$ of $T$ and let $w$ be its neighbor. We identify $z$ with $x$ and $w$ with $y$, and then greedily identify the remaining $k-2$ vertices of $T$ with distinct vertices of $Q\setminus\{y\}$. In this way, we obtain a copy of $T$ in $G$, which is a contradiction. Thus, the $a$ copies of $K_{k-1}$ obtained above are connected components of $G$.

Suppose that $b\geq r$. By Lemma~\ref{lem:packing}, there is exactly one remaining $K_r$-component containing $b$ vertices of $G$. Since the equality holds in~\eqref{eq:upper}, this component contains all $\binom{b}{r}$ copies of $K_r$ on these $b$ vertices. Hence, these vertices induce a copy of $K_b$. These $a+1$ components contain all $n$ vertices of $G$, and therefore $G\cong aK_{k-1}\cup K_b$.

Now suppose that $b<r$. By Lemma~\ref{lem:packing}, every remaining $K_r$-component would contain fewer than $r$ vertices of $G$. However, every $K_r$-component contains a copy of $K_r$, and thus contains at least $r$ vertices. This implies that there are no remaining $K_r$-components.
Since the $a$ copies of $K_{k-1}$ obtained above are connected components of $G$ and contain $a(k-1)$ vertices, the remaining $b$ vertices induce a graph $H$ such that $G\cong aK_{k-1}\cup H$. Conversely, every graph of this form is $T$-free and contains $a\binom{k-1}{r}$ copies of $K_r$. This completes the proof of Theorem~\ref{thm:main}.

\section{\normalsize Concluding remarks}\label{sec:concluding}
We have proved Conjecture~\ref{con:GP} when $r=k-2$ or $r=k-3\geq5$. Combining Theorem~\ref{thm:main} with the immediate case $r=k-1$, we determine the exact value of $\ex(n,K_r,T)$ for $r\in\{k-3,k-2,k-1\}$, where $k\geq8$ is required when $r=k-3$.

The proof for $r=k-3$ relies on Lemmas~\ref{lem:two-cliques} and~\ref{lem:three-outside}, both of which require $r\geq5$, and therefore does not cover the two cases $(k,r)\in\{(6,3),(7,4)\}$. In particular, the conclusion of Lemma~\ref{lem:two-cliques} does not extend to the case $(k,r)=(7,4)$. 
Indeed, the graph formed by two copies of $K_4$ sharing exactly one vertex does not contain the seven-vertex tree obtained by subdividing every edge of $K_{1,3}$ once. Note that this graph is not a counterexample to Conjecture~\ref{con:GP}, since it contains only two copies of $K_4$. It would be interesting to determine whether Conjecture~\ref{con:GP} holds in these two exceptional cases and for all $3\leq r\leq k-4$.

\section*{\normalsize Funding}
The research of Zhou and Yuan was supported by the National Natural Science Foundation of China (Nos.~12271337 and 12371347).

\section*{\normalsize Declaration of interest}
The authors declare no known conflicts of interest.

\end{document}